\documentclass{amsart}

\usepackage{amssymb} % \mathbb
\usepackage{amsthm}
\usepackage{amsmath}
\usepackage{latexsym}
\usepackage{mathrsfs}
\usepackage{enumerate}
\usepackage[dvips]{graphicx}
\usepackage{color}

\definecolor{mgre}{cmyk}{0.92,0.00,0.59,0.25}

\def\bE{{\mathbb{E}}}

\def\bP{{\mathbb{P}}}

\def\bR{{\mathbb{R}}}

\def\cA{{\mathcal{A}}}

\def\cD{{\mathcal{D}}}
\def\cE{{\mathcal{E}}}

\def\sB{{\mathscr{B}}}

\def\sF{{\mathscr{F}}}

\def\1{{\mathbf{1}}}

\def\<{\langle}
\def\>{\rangle}

\newtheorem{thm}{Theorem}[section]
\newtheorem{lem}{Lemma}[section]
\newtheorem{cor}{Corollary}[section]
\newtheorem{prop}{Proposition}[section]
\newtheorem{rem}{Remark}[section]

\newtheorem{exmp}{Example}[section]

\newtheorem{pro}{Problem}[section]

\newcommand{\bd}{\begin{defi}}
\newcommand{\ed}{\end{defi}}

\newcommand{\bpro}{\begin{pro}}
\newcommand{\epro}{\end{pro}}

\newcommand{\bec}{\begin{cases}}
\newcommand{\eec}{\end{cases}}

\newcommand{\bpr}{\begin{prob}}
\newcommand{\epr}{\end{prob}}

\newcommand{\bt}{\begin{thm}}
\newcommand{\et}{\end{thm}}

\newcommand{\ba}{\begin{ass}}
\newcommand{\ea}{\end{ass}}

\newcommand{\br}{\begin{rem}}
\newcommand{\er}{\end{rem}}

\newcommand{\bpm}{\begin{pmatrix}}
\newcommand{\epm}{\end{pmatrix}}

\newcommand{\be}{\begin{ex}}
\newcommand{\ee}{\end{ex}}

\newcommand{\bp}{\begin{prop}}
\newcommand{\ep}{\end{prop}}

\newcommand{\bl}{\begin{lem}}
\newcommand{\el}{\end{lem}}

\newcommand{\bc}{\begin{cor}}
\newcommand{\ec}{\end{cor}}

\newcommand{\bq}{\begin{que}}
\newcommand{\eq}{\end{que}}

\newcommand{\beqn}{\begin{eqnarray*}}
\newcommand{\eeqn}{\end{eqnarray*}}

\newcommand{\beqnn}{\begin{eqnarray}}
\newcommand{\eeqnn}{\end{eqnarray}}

\newcommand{\bequ}{\begin{equation}}
\newcommand{\eequ}{\end{equation}}

\newcommand{\benu}{\begin{enumerate}}
\newcommand{\eenu}{\end{enumerate}}

\newcommand{\barr}{\begin{array}{rcl}}
\newcommand{\ear}{\end{array}}

\newcommand{\la}{\label}

\usepackage{amsmath}	% required for `\cases' (yatex added)

\begin{document}

\title [Existence and uniqueness of quasi-stationary distributions]{Existence and Uniqueness of Quasi-Stationary Distributions for  Symmetric Markov Processes with Tightness Property}

\author{Masayoshi Takeda}
\address{Mathematical Institute,
Tohoku University, Aoba, Sendai, 980-8578, Japan}
\email{takeda@math.tohoku.ac.jp}
\thanks{The author was supported in part by Grant-in-Aid for Scientific
Research (No.26247008(A)) and Grant-in-Aid for Challenging Exploratory Research (No.25610018), Japan Society for the Promotion of Science.}

%    General info
%\date{\today}


\begin{abstract}
Let $X$ be an irreducible symmetric Markov process with the strong Feller property. We assume, in addition, that $X$ is explosive and 
has a tightness property. 
We then prove the existence and uniqueness of quasi-stationary distributions of $X$.
\end{abstract}

\maketitle

\section{Introduction}\label{ra_sec1}
Let $E$ be a locally compact separable metric space and $m$ a positive Radon measure on $E$ with full topological support. 
Let $X=(\Omega,X_t,\bP_x,\zeta)$ be an $m$-symmetric Markov process
(SMP for short) on $E$. Here $\zeta $ is the lifetime
of $X$.  We assume that the process $X$ is irreducible and strong Feller,
in addition, possesses a {\it tightness property}, i.e., for any $\epsilon>0$, there exists a compact set $K$ such that 
$\sup_{x\in E}R_11_{K^c}(x)\leq \epsilon$. Here $1_{K^c}$ is the indicator function of the complement of $K$ and $R_{1}$ 
is the 
1-resolvent of $X$. In this paper, we call the family of SMPs with these three properties {\it Class {(T)}}.
%$R_1f(x)=\bE_x(\int_0^{\infty}e^{-t}f(X_t)dt)$.  

We prove in \cite{T-C} that
if $X$ is in Class {\rm (T)}, then for any $\gamma>0$ there exists
a compact set $K$ such that 
$$
\sup_{x\in E}\bE_x(e^{\gamma \tau_{K^c}})<\infty,
$$
where $\tau_{K^c}$ is the first exit time from $K^c$. As a result, its transition operator $p_t$ is a compact operator on $L^2(E;m)$ and 
all its eigenfunctions have bounded continuous versions (\cite[Theorem 4.3, Theorem 5.4]{T-C}).
If $X$ in Class {\rm (T)} is not conservative, it explodes very fast
 in a sense that the lifetime is exponentially integrable (see (\ref{ee}) below). In particular, $X$ is {\it almost surely killed}, 
$
\bP_x(\zeta<\infty)=1$ for all $x\in E$. 
 The objective of this paper is to prove the existence and uniqueness of quasi-stationary distributions of explosive SMPs in Class (T).
%employing these properties proved in \cite{T-C}.

A probability measure $\nu$ on $E$ is said to be a \textit{quasi-stationary distribution} (QSD for short) of X, if for all $t \ge 0$ and all Borel 
subset $B$ of $E$
\bequ\la{quasi-s}
\nu(B) = \bP_{\nu}(X_t \in B \,|\, t<\zeta),
\eequ
that is, the distribution of $X_t$ conditioned to survive up to $t$ equals $\nu$ over time if the initial distribution $\nu$ is a QSD. 

Let $\phi_0$ be the smallest (principal) eigenfunction of $p_t$ with eigenvalue $\lambda_0$, $p_t\phi_0=e^{-\lambda_0}\phi_0$.
As stated above, we can suppose that $\phi_0$ is a
bounded continuous. Moreover, we can show that $\phi_0$ is strictly positive and integrable, $\phi_0\in L^1(E;m)$ (Lemma \ref{l1}). 
Hence we can define the 
probability measure $\nu^{\phi_0}$ by
$$
\nu^{\phi_0}(B) = \frac{\int_B \phi_0 \, dm}{\int_E \phi_0 \, dm},\ \ B\in\sB(E),
$$
where $\sB(E)$ denotes the totality of Borel subset of $E$. Our main result is as follows (Theorem \ref{main}): If $X$ is in Class (T), then $\nu^{\phi_0}$ is the unique QSD of $X$.

% by the Markov property. 
%We show in Theorem \ref{main} that the measure $\nu^{\phi_0}$ is a unique QSD of $X$.

For the proof of Theorem \ref{main}, the following fact is crucial: Every SMP can be transformed to an ergodic SMP by multiplicative functional.
More precisely, let 
${X}^{\phi_0}=(\Omega,X_t,\bP^{\phi_0}_x,\zeta)$ be the process  transformed by the multiplicative functional,
\begin{equation}
L^{\phi_0} _t=e^{\lambda_0 t}\frac{\phi_0 (X_t)}{\phi_0 (X_0)}1_{\{t<\zeta\} }.
\label{Lt}
\end{equation}
We then see from Lemma 6.3.2 in \cite{FOT} that $X^{\phi_0}
$ is an irreducible, conservative $\phi_0^2m$-SMP on $E$. 
We can prove that $\nu^{\phi_0}$ is a QSD using the $\phi_0^2m$
-symmetry and conservativeness of $X^{\phi_0}$ (Corollary \ref{nu}). 
Applying Fukushima's ergodic theorem (Theorem \ref{fuku-e} below)
 to $X^{\phi_0}$, we can prove that  $\nu^{\phi_0}$ is a unique QSD of $X$. Indeed,  
since $\phi_0$ is strictly positive, 
bounded continuous as remarked above, 
$$
\sup_{x\in E}\frac{1_K}{\phi_0}(x)\leq \frac{1}{\inf_{x\in K}\phi_0(x)}<\infty
$$
for any compact set $K$.  
Hence by Theorem \ref{fuku-e} and Corollary \ref{in-e}, we have
$$
\lim_{t\to\infty}{\mathbb E}^{\phi_0}_x\left(\frac{1_K}{\phi
	_0}(X_t)\right)=\int_K\phi_0dm,\ \ \forall x\in E,
$$
which leads us to the uniqueness of QSD (Theorem \ref{main}).
%There exists a lot of studies for one-dimensional diffusion processes (cf. \cite{CMS}). 

We know that a minimal one-dimensional diffusion process is
in Class (T)   if and only if no natural
boundaries in Feller's classification are present (Example \ref{one}).  
In \cite{CCLMM}, they treat a one-dimensional diffusion process on
$[0,\infty)$ defined as the solution of the SDE:
$$
dX_t=dB_t-q(X_t)dt
%, \ \ q(x)=\frac{1}{2x}-\frac{x}{2}+\frac{x^3}{8}
$$
whose boundaries $0$ and $\infty$ are exit and 
entrance respectively. Theorem \ref{main} says that one-dimensional diffusion processes without natural boundary have a unique 
QSD in general. 

We give two examples 
of multi-dimensional SMPs in Class (T), absorbing Brownian motions on domain {\it thin at infinity} and  
 killed Brownian motions on $\bR^d$, which are treated in \cite{TTT}. 

Finally, we remark that if the semigroup of an explosive symmetric
Markov processes in Class (T) is {\it intrinsic ultracontractive}, $\nu^{\phi_0}$
is a Yaglom limit: for any probability measure $\mu$
$$
\lim_{t\to\infty}\bP_\mu(X_t \in B\,|\,t<\zeta)=\nu^{\phi_0}(B).
$$ 
For example, let $X^D=(\bP^D_x,X_t,\tau_D)$ be an absorbing rotationally symmetric $\alpha$-stable process on bounded open set $D$,
where $0<\alpha<2$ and $\tau_D$ is the first exit time from $D$. We then see that 
$X^D$ is intrinsic ultracontractive (\cite{KT}), and thus 
$$
\lim_{t\to\infty}\bP^D_x\left(X_t\in B\,\mid\,t<\tau_D \right)=\nu^{\phi_0}(B),\ \ \forall B\in\sB(D),
$$ 
which is an extension of a result of Pinsky \cite{Pin} to absorbing symmetric $\alpha$-stable processes.
In \cite[Example 4]{Kw}, he give examples of open sets $D$ such that $m(D)=\infty$ and $X^D$ is intrinsic ultracontractive. 
Applications of the intrinsic ultracontractivity to the Yaglom limit were studied in \cite{KP}, \cite{miura}.

\section{Ergodic properties of SMPs}
In this section, we summarize results on ergodic properties of SMPs. 
Let $E$ be a locally compact separable metric space and $E_\triangle $ the one-point 
compactification of $E$ with adjoined
point $\triangle $. Let $m$ be a positive Radon measure on $E$ with full topological support. 
Let $X=(\Omega,X_t,\bP_x,\zeta)$ be an $m$-SMP.  Here $\zeta$ is the lifetime of $X$, $\zeta=\inf\{t>0:X_t=\Delta\}$. 
Denote by 
$\{p_t;{t\geq 0}\}$ and $\{R_{\alpha};{\alpha>0}\}$ the semigroup and resolvent of $X$: 
$$
p_tf(x)=\bE_x(f(X_t)),\quad R_{\alpha}f(x)=\bE_x\left(\int_0^{\infty}e^{-\alpha t}f(X_t)dt\right).
$$
In this section, we further assume that $X$ is conservative, $\bP_x(\zeta=\infty)=1$, and satisfies  

\medskip
	{\bf (I) (Irreducibility)} \ If a Borel set $A$ is $p_t$-invariant, that is, 
	$p_t(1_Af)(x)=1_Ap_tf(x)$\ $m$-a.e. for any $f\in
	L^2(E;m)\cap b{\mathscr B}(E)$  
	and $t>0$, then $A$ satisfies either $m(A)=0$ or $m(E\setminus A)=0$. 
	Here $b{\mathscr B}(E)$ is the space of 
	bounded Borel functions on
	$E$. 

\medskip
%if a Borel set $A$ satisfies $p_t(1_Af)(x)=1_Ap_tf(x),\ m$-a.e. 
%for any $f\in L^2(E;m)\cap {b\sB}(E)$ and any $t>0$, then $m(A)=0$ or $m(E\setminus A)=0$. Here $b\sB(E)$ is the set of bounded Borel functions
%on $E$.
The symmetry of $X$ enables us to strengthen the ergodic theorem as follows: Suppose $m(E)<\infty$. For $f\in L^\infty(E;m)$
\begin{equation}\la{s-e}
p_tf(x)\rightarrow \frac{1}{m(E)}\int_Ef(x)dm,\quad m\text{-a.e.}\ x.
\end{equation}
Following the argument in \cite{F-erg}, we will give a proof of (\ref{s-e}).
%On account of the symmetry, we have the ergodic theorem due to Fukushima \cite{F-erg}.

\bt 
Suppose $m(E)<\infty $. For any $f\in L^\infty(E;m)$, 
there exists a function $g$ in $L^\infty(E;m)$ such that 
$$
\lim_{t\to\infty}p_tf=g,\ \ m\text{-a.e. and in} \ L^1(E;m).
$$
Moreover, $g$ is $p_t$-invariant, $p_tg=g,\ m\text{-a.e.}$
\la{erg}
\et

\begin{proof}
	Define ${\mathscr  G}_t=\sigma \{X_s\mid s\geq t\}$ and $Y_t=\bE_m(f(X_0)|{\mathscr  G}_t)$, where $\bP_m(\cdot)=\int_E\bP_x(\cdot)dm(x)$.
	By the time reversibility of $X_t$ with respect to $\bP_m$, $Y_t=p_tf(X_t)$, $\bP_m$-a.e., 
	and so
	$$
	\bE_m(Y_t|\sF_0)=\bE_m(p_tf(X_t)|\sF_0)=p_{2t}f(X_0),\ \ \bP_m\text{-a.e. }
	$$
Here $\sF_0=\sigma \{X_0\}$.
	Since $f(X_0)\in L^1(\bP_m)$ and $Y_t$ is a reversed martingale, 
	$$
	\lim_{t\to\infty }Y_t=\bE_m(f(X_0)|\cap_{t>0}{\mathscr  G}_t), \ \bP_m\text{-a.e. and in}\ L^1(\bP_m)
	$$
(cf. \cite[Theorem:II.51.1]{RW}). Put $Z=\bE_m(f(X_0)|\cap_{t>0}{\mathscr  G}_t)$. 
	Noting that $|Y_t|\leq \|f\|_\infty,\ \bP_m\text{-a.e.}$ by 
	the definition of $Y_t$, we see from the conditional bounded convergence theorem (cf. \cite[II.40.41, (41)(g)]{RW}) that	
	\begin{align}\la{ue}
	\lim_{t\to\infty }p_{2t}f(X_0)&=\lim_{t\to\infty }\bE_m(Y_t|\sF_0)\\
	&=\bE_m(Z|\sF_0)=\bE_{X_0}(Z),\ \ \bP_m\text{-a.e. and in} \ L^1(\bP_m).\nonumber
	\end{align}
	Put $g(x)=\bE_x(Z)$. We then see from (\ref{ue}) that $\lim_{t\to\infty}p_tf=g$, $m$-a.e. and  in $L^1(E;m)$. The $p_t$-invariance of $g$ follows from 
	$$
	p_tg=\lim_{s\to\infty}p_t(p_sf)=\lim_{s\to\infty}p_{t+s}f=g,\ \ m\text{-a.e.},
	$$
which completes the proof.
\end{proof}

%Though the function $f$ in Theorem \ref{fuku-e} is suppose to be bounded
%in \cite{Fu},this condition can be easily removed by approximating bounded $L^1$-functions.
\begin{thm}{\rm (\cite{F-erg})} \label{fuku-e} Assume $m(E)<\infty $.  
	If the Markov process $X$ is irreducible and conservative, then for $f\in L^\infty(E;m)$
	\bequ\la{const}
	\lim_{t\to\infty }p_{t}f(x)=\frac{1}{m(E)}\int_Efdm, \ \ m\text{-a.e. and in} \ L^1(E;m)
	\eequ
	\label{ae-limit}
\end{thm}
\begin{proof}
	By combining Theorem  \ref{erg} with \cite[Theorem 2.1.11]{CF}(see also \cite[Theorem 1]{K}), we see $\lim_{t\to\infty }p_{t}f$ is constant $m$-a.e.
	Since $(p_{t}f,1)_m=\int_Efdm$, the constant is equal to the right-hand side of (\ref{const}).
\end{proof}

\br \la{any} \rm
Suppose that $X$ satisfies the absolute continuity condition:

\medskip
{\bf (AC)} \  $p_t(x,dy)=p_t(x,y)m(dy)$, $\forall t>0,\ \forall x\in E$.

\medskip
\noindent
Then ``$m$-a.e. $x$" in Theorem \ref{fuku-e} can be strengthened to ``all $x$". Indeed, for any $x\in E$
\vspace{-2mm}
\begin{align*}
	\lim_{t\to\infty }p_tf(x)&=	\lim_{t\to\infty }\int_Ep_1(x,y)\left(\int_Ep_{t-1}(y,z)f(z)dm(z) \right) dm(y)\\
	&=\int_Ep_1(x,y)\lim_{t\to\infty }\left(\int_Ep_{t-1}(y,z)f(z)dm(z) \right) dm(y)\\
	&=\int_Ep_1(x,y)\left(\frac{1}{m(E)}\int_Efdm\right) dm(y)=\frac{1}{m(E)}\int_Efdm.
\end{align*}

\er

\begin{cor}\label{in-e} Suppose the assumptions of Theorem \ref{fuku-e} hold. Assume, in addition, {\bf (AC)} and 
	the ultracontractivity, $\|p_t \|_{1,\infty} \leq c_t<\infty$. 
	Here  $\|\cdot\|_{1,\infty}$ is the operator norm from 
	$L^1(E;m)$ to $L^\infty(E;m)$. Then for $f\in L^1(E;m)$ 
	\bequ\la{ka}
	\lim_{t\to\infty }p_{t}f(x)=\frac{1}{m(E)}\int_Efdm, \ \ \forall x\in E.
	\eequ
\end{cor}
\begin{proof}
For $f\in L^1(E;m)$, $p_1f\in L^\infty(E;m)$ by the ultracontractivity. Hence 
	$$
	\lim_{t\to\infty }p_{t}f(x)=	\lim_{t\to\infty }p_{t-1}(p_1f)(x)=\frac{1}{m(E)}\int_Ep_1fdm,\ \ \forall x\in E
	$$	
	by Theorem \ref{fuku-e} and Remark \ref{any}. By the symmetry 
	with respect to $m$ 
	and conservativeness of $p_1$
	$$
	\int_Ep_1fdm=\int_Ep_11\cdot fdm=\int_Efdm, 
	$$
	 and (\ref{ka}) is proved.
\end{proof}

\section{Quasi-stationary distribution}\label{sec:quasi stationary}
In this section, we consider the existence and uniqueness of quasi-stationary distributions. 
%as an application of Fukushima's ergodic theorem.
We assume that $X$ possesses the next three properties:

\smallskip
%\noindent
\begin{enumerate}
	\renewcommand{\theenumi}{$({\bf \Roman{enumi}})$}
	\renewcommand{\labelenumi}{\theenumi}
	\item\label{item:1} %[\textup{$({\bf I})$}]
	({\bf Irreducibility}) 
	%\ If a Borel set $A$ is $p_t$-invariant, i.e., 
%	$p_t(1_Af)(x)=1_Ap_tf(x)$\ $m$-a.e. for any $f\in
%	L^2(E;m)\cap b{\mathscr B}(E)$  
%	and $t>0$, then $A$ satisfies either $m(A)=0$ or $m(E\setminus A)=0$. 

\smallskip	
	\item\label{item:2} %[\textup{$({\bf II})$}] 
	({\bf Strong Feller Property}) \  For each $t>0$,
	$p_t(b{\mathscr  B}(E))\subset bC(E)$, 
	where $bC(E)$ is the space of bounded continuous functions on $E$.
	
	\smallskip	
	\item\label{item:3} %[\textup{$({\bf III})$}]
	({\bf Tightness})  
	\ For any $\epsilon>0$, there exists a compact set $K$ such that 
	\begin{equation*}
	\sup_{x\in E}R_11_{K^c}(x)\leq \epsilon.
	\end{equation*}
%	Here $1_{K^c}$ is the indicator function
%	of the complement of the compact set $K$. 
\end{enumerate}
A SMP with three properties above is 
said to be in {\it Class (T)}. Note that Conditon ({\bf II}) implies {\bf (AC)}.

We see that if $X$ is not conservative, 
the tightness property 
implies a fast explosion in a sense that the lifetime $\zeta $ is exponentially integrable. In particular,
$X$ is almost surely killed, $\bP_x(\zeta<\infty)=1$ for any $x\in E$. Indeed, let  $({\mathcal E},\cD(\cE))$ be
the Dirichlet form on $L^2(E;m)$ generated by
$X$: 
\begin{equation}\label{D-form}
\left\{
\begin{split}
& \cD(\cE)=\left\{u\in L^2(E;m)\, \Big|\, \lim_{t\to
	0}\frac{1}{t}(u-T_tu,u)_m<\infty \right\} \\
& \cE(u,v)=\lim_{t\to 0}\frac{1}{t}(u-T_tu,v)_m.
\end{split}
\right.
\end{equation}
We define 
$$
\lambda _0= \inf \{ \cE(u,u)\, \mid\, u \in \cD(\cE),\ \| u \|_2 =1 \},
$$
where $\| \cdot \|_2$ is the $L^2(E;m)$-norm.  
We then see in \cite[Corollary 3.8]{T8} that $\lambda _0>0$ and 
for $0<\gamma<\lambda_0$
\bequ\la{ee}
\sup_{x\in E}\bE_x(e^{\gamma \zeta})<\infty.
\eequ
%In particular, $X$ is surely killed;
%$$
%\bP_x(\zeta<\infty)=1,\ \forall x\in E.
%$$
In the sequel, we assume that $X$ is a explosive SMP in Class (T).
%%%%%%%%%%%%%%%%%%%%%%%%%%%%%%%%%%%%%%%%%%%%%%%%%%%%%%%%%%%%%%%%%%%%%%%%%

\medskip
A probability measure $\nu$ on $E$ is said to be \textit{quasi-stationary distribution} (QSD for short) of X if for all $t \ge 0$ and all Borel set $B\in \sB(E)$, 
\[
\nu(B) = \bP_{\nu}(X_t \in B \,|\, t<\zeta) \left( = \frac{\bP_{\nu}(X_t \in B)}{\bP_{\nu}(X_t\in E)} \right),
\]
where $\bP_\nu(\cdot)=\int_E\bP_x(\cdot)d\nu(x)$.
QSDs capture the long-time behavior of surely killed process $X$ 
when $X$ is conditioned to survive. 

A function $\phi_0$ on $E$ is called a \textit{ground state} of $(\cE, \cD(\cE))$
if $\phi_0 \in \cD(\cE),\, \| \phi_0 \|_2 = 1$ and $\lambda _0=\cE(\phi_0, \phi_0)$.
The ground state $\phi_0$ exists because the embedding of 
$(\cE_1, \cD(\cE))$ into $L^2(E;m)$ is compact
 (\cite[Theorem 2.1]{T-G}). Here $\cE_1=\cE+(\ ,\ )_m$.

\bl\la{kl}
For a Borel set $B\subset E$ with $m(B)>0$, define
\bequ\la{k}
\lambda^B _0= \inf \left\lbrace \cE(u,u)+\int_Bu^2dm\, \Big|\, u \in \cD(\cE),\ \| u \|_2 =1 \right\rbrace .
\eequ
Then it holds that $\lambda^B _0>\lambda _0$.
\el
\begin{proof}
	There exists a minimizer $\phi^B_0$ attaining the infimum in (\ref{k}) by \cite[Theorem 2.1]{T-G}.
	Hence
	\begin{align*}
	\lambda^B _0&= \cE(\phi^B_0,\phi^B_0)+\int_B(\phi^B_0)^2dm>\cE(\phi^B_0,\phi^B_0)\geq	\cE(\phi_0,\phi_0)=\lambda_0.
	\end{align*}
\end{proof}

\begin{prop}\la{b-c}{\rm (\cite{T-C})}
	The ground state $\phi_0$ has a bounded continuous version with 
	$\phi_0(x)>0$ for any $x\in E$. 
\end{prop}

For a compact set $K$ with $m(K)>0$, define
\begin{align*}
p^K_tf(x)&=\bE_x\left( e^{-\int_0^t1_K(X_s)ds}f(X_t)\right),\ t\geq 0,\\  	
R_\beta^{\lambda_0,K}f(x)&=\bE_x\left( \int_0^\zeta e^{-\beta t+\lambda_0t}p^K_tf(x)dt\right),\ \beta\geq 0.
\end{align*}
We denote $R^{\lambda_0,K}$ for $R_0^{\lambda_0,K}$ simply.
\bl\la{fini}
It holds that 
$
\sup_{x\in E}R^{\lambda_0,K}1(x)<\infty.
$
\el

\begin{proof}
	By the $L^p$-independence of the growth bound of $p^K_t$
	(\cite[Theorem 1.3]{CKK}),   for any $\delta>0$ there exists a positive constant $C(\delta)$ 
	such that
	$$
	\sup_{x\in E}p^K_t1(x)=\|p^K_t1\|_\infty\leq C(\delta)e^{-(\lambda^K_0-\delta)t}.
	$$
	Since $\lambda^K _0>\lambda_0$ by Lemma \ref{kl}, for $0<
	\delta<\lambda^K _0-\lambda_0$
	\begin{align*}
	\|R^{\lambda_0,K}1\|_\infty&\leq\int_0^\infty e^{\lambda_0t}\sup_{x\in E}p^K_t1(x)dt	
	\leq C(\delta)\int_0^\infty e^{-(\lambda^K_0-\lambda_0-\delta)t}dt\\
	&=\frac{C(\delta)}{\lambda^K_0-\lambda_0-\delta}<\infty.
	\end{align*}
\end{proof} 

We define symmetric bilinear forms on $L^2(E;m)$: For $u\in\cD(\cE)$
\begin{align*}
\cE^{\lambda_0}(u,u)&=\cE(u,u)-\lambda_0\int_Eu^2dm, \\
\cE^{\lambda_0,K}(u,u)&=\cE(u,u)-\lambda_0\int_Eu^2dm+\int_Eu^21_Kdm.
\end{align*}
For a general symmetric bilinear form $\cA$, $\cA_\beta$ denotes $\cA+\beta(\ ,\ )_m$. 	

\bl\la{b-e}
The ground state $\phi_0$ satisfies $\phi_0(x)=R^{\lambda_0,K}(\phi_01_K)(x)$ for all $x\in E$.
\el
\begin{proof}
	For $\varphi\in b\sB_0^+(E)$, the set of non-negative bounded functions with compact support,
	\begin{align*}
	\cE^{\lambda_0,K}(R^{\lambda_0,K}_\beta\varphi,R^{\lambda_0,K}_\beta\varphi)
	&\leq\cE^{\lambda_0,K}_\beta(R^{\lambda_0,K}_\beta
	\varphi,R^{\lambda_0,K}_\beta\varphi)\\
	&=\int_E\varphi R^{\lambda_0,K}_\beta
	\varphi dm\leq\int_E\varphi R^{\lambda_0,K}
	\varphi dm<\infty
	\end{align*}
	by Lemma \ref{fini}. Since $R^{\lambda_0,K}_\beta\varphi\uparrow R^{\lambda_0,K}\varphi$ as $\beta\downarrow 0$,
	 the function $R^{\lambda_0,K}\varphi$ belongs to 
	the extended Schr\"odinger space $\cD_e(\cE^{\lambda_0,K})$ (For the definition of extended Schr\"odinger space, see \cite[Section 2]{T-I}).

By the definition of $\cE^{\lambda_0,K}$,
$$
	\cE^{\lambda_0,K}(\phi_0,R^{\lambda_0,K}_\beta\varphi)= \cE^{\lambda_0}(\phi_0,R^{\lambda_0,K}_\beta\varphi) +\int_E1_K\phi_0. R^{\lambda_0,K}_\beta\varphi dm.
$$	
Noting that $\cD_e(\cE^{\lambda_0,K})\subset\cD_e(\cE^{\lambda_0})$
because $\cE^{\lambda_0}(u,u)\leq\cE^{\lambda_0,K}(u,u)$, we have 
		\bequ\la{al}
	\cE^{\lambda_0,K}(\phi_0,R^{\lambda_0,K}\varphi)= \cE^{\lambda_0}(\phi_0,R^{\lambda_0,K}\varphi) +\int_E1_K\phi_0 R^{\lambda_0,K}\varphi dm
	\eequ
as $\beta\to 0$.
%	Let $\psi_n=\phi_0\wedge n$. Then 
%	$\cE_1(\psi_n-\phi_0,\psi_n-\phi_0)\longrightarrow 0$ as $n\to\infty$, and 
%	\begin{align*}
%	\cE^{\lambda_0,K}(\phi_0,R^{\lambda_0,K}\varphi)&= \cE^{\lambda_0}(\phi_0,R^{\lambda_0,K}\varphi) +\int_E1_K\phi_0 R^{\lambda_0,K}
%	\varphi dm
%	\end{align*}

	Since $\phi_0$ is the eigenfunction corresponding to $\lambda_0$, $\cE^{\lambda_0}(\phi_0,R_\beta^{\lambda_0,K}\varphi)=0$
	for any $\beta >0$, and so $\cE^{\lambda_0}(\phi_0,R^{\lambda_0,K}\varphi)=0$. Hence, by (\ref{al}) and the symmetry of $R^{\lambda_0,K}$
	with respect to $m$
	\bequ\la{fir}
	\cE^{\lambda_0,K}(\phi_0,R^{\lambda_0,K}\varphi)=\int_E1_K\phi_0 R^{\lambda_0,K}
	\varphi dm=\int_ER^{\lambda_0,K}(1_K\phi_0 )
	\varphi dm.
	\eequ
%	by the symmetry of $R^{\lambda_0,K}$ with respect to $m$.
	
	On the other hand,
	\begin{align}\la{b-l}
	\cE^{\lambda_0,K}(\phi_0,R^{\lambda_0,K}_\beta\varphi)&= \cE^{\lambda_0,K}_\beta(\phi_0,R^{\lambda_0,K}_\beta\varphi) -\beta\int_E\phi_0 
	R^{\lambda_0,K}_\beta\varphi dm  \\
	&=\int_E\phi_0\varphi dm-\beta\int_E\phi_0 R^{\lambda_0,K}_\beta\varphi dm.\nonumber
	\end{align}
	Since 
	$$
	\int_E\phi_0 R^{\lambda_0,K}_\beta\varphi dm= \int_ER^{\lambda_0,K}_\beta\phi_0 \varphi dm\leq \|\phi_0\|_\infty\|R^{\lambda_0,K}1\|_\infty\int_E\varphi dm<\infty,
	$$
	by Lemma \ref{fini}, we have from (\ref{b-l})
		\bequ\la{sec}
	\cE^{\lambda_0,K}(\phi_0,R^{\lambda_0,K}\varphi)=\int_E\phi_0\varphi dm. 
	\eequ 
	by letting $\beta\to 0$.
	%Taking a sequence $\psi_n=\phi_0\wedge n$ as above, we 
%	see

	By (\ref{fir}) and (\ref{sec}),
	\beqn
\int_{E} R^{\lambda_0,K}(\phi_0 1_K)
	\varphi dm=\int_{E} \phi_0 \varphi dm, \,\  \forall\varphi\in b\sB_0^+(E) 
	\eeqn
	and thus
	\beqn
	\phi_0 =  R^{\lambda_0,K}(\phi_0  1_K),\
	\ m\textrm{-a.e.}
	\eeqn
	By the continuity of both functions,  ``$m$-a.e. $x$" can be strengthen 
	to ``all $x$" .
	
\end{proof}

%Moreover, we can prove 
\bl\la{l1}
The ground state $\phi_0$ belongs to $L^1(E;m)$.
\el

\begin{proof}
	By Lemma \ref{b-e} and the symmetry of $R^{\lambda_0,K}$ with respect to $m$, we see 
	$$
	\int_E\phi_0dm=\int_ER^{\lambda_0,K}(1_K\phi_0)dm=\int_E1_K\phi_0R^{\lambda_0,K}1dm.
	$$
The right hand side is finite by Proposition \ref{b-c} and Lemma \ref{fini}.	
\end{proof}

\bl\la{ab-c}
Let $\mu$ be a QSD. Then $\mu$ is absolutely continuous with respect to
$m$.
\el
\begin{proof} If $m(B)=0$, then
	$$
	\bP_{\mu}(X_t \in B)=\int_E\left( \int_Bp_t(x,y)dm(y)\right) d\mu=0,
	$$
	and thus $\mu(B)=\bP_{\mu}(X_t \in B)/\bP_{\mu}(t<\zeta)=0$.
\end{proof}

%%%%%%%%%%%%%%%%%%%%%%%%%%%%%%%%%%%%%%%%%%%%%%%%%%%%%%%%%%%%%%%%%%%%%%%%%
We define the space ${\mathcal D}^+(A)$ by
$$
{\mathcal D}^+(A)=\left\{R_{\alpha}f\,\mid\, \alpha>0,\ f\in L^2(E;m)\cap bC^+(E),\ \ 
f\not\equiv 0\right\}.
$$
Here $bC^+(E)$ is the set of non-negative bounded continuous functions.
 For $\phi=R_{\alpha}g\in{\mathcal D}^+(A)$ define the multiplicative functional $L^\phi$ by
\begin{equation}	\label{L}
L^\phi _t=\frac{\phi (X_t)}{\phi (X_0)}\exp\left(-\int_0^t\frac{A\phi }{\phi }(X_s)ds\right)1_{\{t<\zeta\} },\ \ A\phi =
\alpha \phi-g.
\end{equation}
Let ${X}^{\phi}=(\Omega,X_t,\bP^{\phi}_x,\zeta)$ the transformed process of $X$ by $L^{\phi}_t$ and denote by $p_t^{\phi}$ its semigroup, $p_t^{\phi}f(x)=\bE_x(L^{\phi}_tf(X_t))$. 
We then see from Lemma 6.3.2 in \cite{FOT} that $X^{\phi}
$ is an irreducible, conservative $\phi^2m$-SMP on $E$, 
$(p_t^{\phi}f,g)_{\phi^2m}=(f, p_t^{\phi}g)_{\phi^2m}$ (In
\cite{CFTYZ} this fact is extended to $\phi\in\cD(\cE)$). 
Since $\phi_0\in{\mathcal D}^+(A)$ and $A\phi_0=-\lambda_0\phi_0$, $L^{\phi_0}_t$ in (\ref{L}) is simply written as 
\begin{equation}
L^{\phi_0} _t=e^{\lambda_0 t}\frac{\phi_0 (X_t)}{\phi_0 (X_0)}1_{\{t<\zeta\} }.
\label{Lt}
\end{equation}
Hence the following equalities hold: 
\begin{align}
p^{\phi _0}_tf(x)&=e^{\lambda _0t}\frac{1}{\phi _0(x)}{\mathbb
	E}_x\left(\phi _0(X_t)f(X_t)\right)=e^{\lambda _0t}\frac{1}{\phi _0(x)}p_t(\phi_0f)(x)\nonumber
\label{phi-semi}
\end{align}
and so
\begin{equation}
p_tf(x)=e^{-\lambda _0t}{\phi _0(x)}p_t^{\phi_0}\left(\frac{f}{\phi _0}\right)(x). 
\label{j-phi-semi}
\end{equation}

We see from Lemma \ref{l1} that the probability measure $\nu^{\phi_0}$ can be defined by
\begin{equation}\label{qsd}
\nu^{\phi_0}(B) = \frac{\int_B \phi_0 \, dm}{\int_E \phi_0 \, dm}.
\end{equation}

\bl\la{nu}
The measure $\nu^{\phi_0}$ is a QSD of X. 	
\el
\begin{proof} 	
By (\ref{j-phi-semi}) 
\begin{align*}
\bP_{\nu^{\phi_0}}(X_t\in B)&=\frac{\int_E\bP_x(X_t\in B)\phi_0(x)dm}{\int_E\phi_0(x)dm}\\
&=\frac{e^{-\lambda_0t}\int_E\bE_x^{\phi_0}{((1_B/\phi_0)(X_t))\phi_0^2(x)dm}}{\int_E\phi_0(x)dm}.
\end{align*}
Since $X^{\phi_0}$ is $\phi_0^2m$-symmetric and conservative, $p^{\phi_0}_t1=1$,
\begin{align*}
&\int_E\bE_x^{\phi_0}{((1_B/\phi_0)(X_t))\phi_0^2(x)dm}=\int_Ep^{\phi_0}_t(1_B/\phi_0)(x)\phi_0^2(x)dm\\
&\qquad\qquad =\int_E(1_B/\phi_0)(x)p^{\phi_0}_t1(x)\phi_0(x)^2dm=
\int_B \phi_0 \, dm.
\end{align*}
Hence we see 
$$
\bP_{\nu^{\phi_0}}(X_t\in B\,|\,X_t\in E)=\frac{\bP_{\nu^{\phi_0}}(X_t\in B)}{\bP_{\nu^{\phi_0}}(X_t\in E)}=\frac{\int_B \phi_0 \, dm}{\int_E \phi_0 \, dm}=\nu^{\phi_0}(B).
$$
\end{proof}
%$\nu^{\phi_0}$ is a QSD of X. 

%%%%%%%%%%%%%%%%%%%%%%%%%%%%%%%%%%%%%%%%%%%%%%%%%%%%%%%%%%%%%%%%%%%%%%%%%

%(see \cite[Hypothesis (H1)$\sim$(H5)]{Cattiaux-Collet-Lambert-Martinez-SanMartin(09)}).
%For example, one of these assumptions requires that $q$ is differentiable.

\begin{thm}\la{main}
	Assume that $X$ is an explosive SMP in Class (T). 
	Then the measure $\nu^{\phi_0}$ defined in (\ref{qsd}) is the  unique QSD of $X$.
\end{thm}

\begin{proof} 
	Let $\mu$ is a QSD, i.e. 
	\begin{equation*}
	\mu(B)={\mathbb P}_\mu\left( X_t\in
	B \mid t<\zeta\right)=\frac{{\mathbb P}_\mu\left( X_t\in
		B\right)}{{\mathbb P}_\mu\left( X_t\in
		E\right)}. 
	\end{equation*}
	For compact sets $K,  F\subset E$,
	\begin{align*}
	\mu(K)&=\frac{{\mathbb P}_\mu\left( X_t\in
		K\right)}{{\mathbb P}_\mu\left( X_t\in E\right)}
	\leq\frac{{\mathbb P}_\mu\left( X_t\in
		K\right)}{{\mathbb P}_\mu\left( X_t\in F\right)}=\frac{\int_Ep_t1_Kd\mu}{\int_Ep_t1_Fd\mu}.
	\end{align*}
	By (\ref{j-phi-semi}), the right hand side equals
	$$
	\frac{\int_E\phi_0p_t^{\phi_0}\left(\frac{1_K}{\phi
			_0}\right) d\mu}{\int_E\phi_0p_t^{\phi_0}\left(\frac{1_F}{\phi
			_0}\right) d\mu}.
	$$
	
	Since 
	$$
	\frac{1_K(x)}{\phi_0(x)}\leq \frac{1}{\inf_{x\in K}\phi_0(x)}<\infty,
	$$
	${1_K}/{\phi_0}$ belongs to $L^\infty(E;m)$. 
	Noting $X^{\phi_0}$ satisfies {\bf (AC)} by definition, we see from 
	Remark \ref{any} that
	$$
	\lim_{t\to\infty}p_t^{\phi_0}\left(\frac{1_K}{\phi
		_0}\right)(x) =\int_K\phi_0dm,\ \ \forall x\in E.
	$$
Hence 
	\begin{equation}\label{as} 
	\lim_{t\to\infty}\int_E\phi_0(x)p_t^{\phi_0}\left(\frac{1_K}{\phi
		_0}\right)(x) d\mu(x)=\int_E\phi_0d\mu\int_K\phi_0dm,		
	\end{equation}
	and 
	\begin{align*}
	\mu(K)&\leq \lim_{t\to\infty}\frac{{\mathbb P}_\mu\left( X_t\in
		K\right)}{{\mathbb P}_\mu\left( X_t\in F\right)}=\frac{\int_K\phi_0dm}{\int_F\phi_0dm}.
	\end{align*}
	By letting $F\uparrow E$, $\mu(K)\leq\nu^{\phi_0}(K)$ and by the inner regularity $\mu(B)\leq\nu^{\phi_0}(B)$ for any $B\in\sB(E)$. 
	Noting that
	$$
	\mu(B)=1-\mu(B^c)\geq 1-\nu^{\phi_0}(B^c)=\nu^{\phi_0}(B),
	$$ 
	we can conclude that $\mu=\nu^{\phi_0}$.
\end{proof}

\bl
If $X$ is intrinsic ultracontractive, then
$$
\lim_{t\to\infty}\bP_\nu(X_t \in B\,|\,t<\zeta)=\nu^{\phi_0}(B)
$$
 for any probability measure $\nu$.
\el
\begin{proof}
	Since
	$$
	\bP_\nu(X_t \in B)=e^{-\lambda_0t}\int_E\phi_0p_t^{\phi_0}\left( \frac{1_B}{\phi_0}\right)d\nu,
	$$
	we have 
	$$
	\bP_\nu(X_t \in B\,|\,t<\zeta)=\frac{\int_E\phi_0p_t^{\phi_0}\left( \frac{1_B}{\phi_0}\right) d\nu}{\int_E\phi_0p_t^{\phi_0}\left( \frac{1}{\phi_0}\right) d\nu}.
	$$
Noting that $1/\phi_0\in L^1(\phi_0^2m)$, we have this lemma by  Corollary \ref{in-e}.	
\end{proof}

\begin{exmp} \la{one} \rm 
	Let us consider a one-dimensional diffusion process $X=(X_t, \bP_x,
	\zeta)$ on an open interval $I=(r_1,r_2)$ 
	such that $\bP_x(X_{\zeta-}=r_1\ {\rm or}\ r_2,\
	\zeta<\infty)=\bP_x(\zeta<\infty),\ x\in I,$ and $\bP_a(\sigma_b<\infty)>0$  
	for any $a, b\in I.$    The diffusion $X$ is symmetric with
	respect to its canonical measure $m$ and it satisfies {\rm I} 
	and {\rm II}. 
	The boundary point $r_i$ of $I$ is classified 
	into four classes: {\em regular boundary, exit boundary, entrance
		boundary and natural boundary\/} (\cite[Chapter 5]{Ito}): 
	\begin{enumerate}  
		\item[\textup{(a)}]  If $r_2$ is a regular 
		or exit boundary, then $\lim_{x\to r_2 }R_11(x)=0$. 
		
		\item[\textup{(b)}]  If $r_2$ is an entrance boundary, then $
		\lim_{r\to r_2}\sup_{x\in I}R_11_{(r,r_2)}(x)=0.
		$
		
		\item[\textup{(c)}] If $r_2$ is a natural boundary, then 
		$\lim_{x\to r_2}R_11_{(r,r_2)}(x)=1$ and thus \\
		$\sup_{x\in(r_1,r_2)}R_11_{(r,r_2)}(x)=1$.
	\end{enumerate}
	Therefore, the tightness property {\rm III} is fulfilled if and only if no natural
	boundaries are present. 
	For the intrinsic ultracontractivity of one-dimensional diffusion processes, refer to \cite{Tomi}. 		
\end{exmp}

\begin{exmp}\rm
Let $\cD$ be the family of open sets in $\bR^d$. We set
\begin{equation*}
\cD_0=\left\{D\in\cD\, \Big|\, \lim_{x\in D,|x|\to\infty} m(D\cap B(x,1))=0\right\},
\end{equation*}
where $m$ denotes the Lebesgue measure on $\bR^d$ and $B(x,1)$ the open ball with center $x$ and radius 1. 
Let $X$ be the symmetric 
$\alpha$-stable process, the Markov process generated by $(-\Delta)^{\alpha/2}\,(0<\alpha\leq 2)$. 
We can show by the same argument as in 
\cite[Lemma 3.3]{TTT} that if an open set $D$
belongs to $\cD_0$, then the absorbing process $X^D$ on $D$ is in Class \textrm{(T)}. 
For the intrinsic ultracontractivity of $X^D$, refer to \cite{A}, \cite{KT}, \cite{Kw}. In particular, it is shown in \cite{KT} that for $0<\alpha<2$
$X^D$ is intrinsic ultracontractive for any bounded open set $D$. As a result, 
$$
\lim_{t\to\infty}\bP^D_x\left(X_t\in B\,\mid\,t<\tau_D \right)=\nu^{\phi_0}(B),\ \ \forall B\in\sB(D).
$$ 
%which is an extension of a result of Pinsky \cite{Pin} to absorbing symmetric $\alpha$-stable processes.
 In \cite[Example 4]{Kw}, the author gives an example of open set $D$ such that $m(D)=\infty$ and $X^D$ is intrinsic ultracontractive. 
\end{exmp}

\begin{exmp}\rm
Let $V$ be a positive function in the local Kato class. 
If 
$$
\lim_{|x|\to\infty}m(\{x\in\bR^d\mid V(x)\leq M\})=0\ \text{for any}\ 
M>0,
$$
then the subprocess of the BM by $\exp\left( -\int_0^tV(B_s)ds\right) $ is in Class (T)(\cite{TTT}).
For the intrinsic ultracontractivity of Schr\"odinger semi-groups, refer to \cite{KK}.
\end{exmp}

\bibliographystyle{amsplain}

\begin{thebibliography}{30}
	\bibitem{A} Aikawa, H.: Intrinsic ultracontractivity via capacitary width, Rev. Mat. Iberoam. {\bf 31}, 1041-1106 (2015).
	%\bibitem{BG} Blumenthal, R.M., Getoor, R.K.: Markov Processes and Potential Theory, Academic Press 
	
	%\bibitem{Br}Brancovan, M.:
	% Fonctionnelles additives speciales des processus recurrents au sens de Harris,  
	% Z. Wahrsch. Verw. Gebiete. {\bf 47}, 163-194 (1979). 
	
	\bibitem{CCLMM} Cattiaux, P., Collet, P., Lamberrt, A., Martinez, S., Meleard, S. San Martin, J.: Quasi-stationary distributions and diffusion models in population dynamics. Ann. Probab. {\bf 37}, 1926-1969 (2009). 
	
	
	\bibitem{CMS} Collet, P., Mart\'inez, S., San Martin, J.: Quasi-stationary distributions,
	Markov chains, diffusions and dynamical systems, Springer (2013).	
	%\bibitem{CMF} Chen, M.-F.: Speed of stability for birth-death processes. Front. Math. China {\bf 5}, 379-515
	%(2010).
	%\bibitem{C} Chen, Z.-Q. : Gaugeability and Conditional  Gaugeability,
	% Trans. Amer. Math. Soc. {\bf 354}, 4639-4679 (2002).
	
	\bibitem{CFTYZ} Chen, Z.-Q., Fitzsimmons, P. J., Takeda, M., Ying, J., Zhang, T.-S. 
	Absolute continuity of symmetric Markov processes, Ann. Probab. {\bf 32}, 2067-2098 (2004).
	
	\bibitem{CF}Chen  Z.-Q. , Fukushima, M.: 
	Symmetric Markov Processes, Time Change and 
	Boundary Theory, London Mathematical Society Monographs 
	Series, {\bf 35}, Princeton University Press (2012).
	
	%\bibitem{CK} Chen, Z.-Q., Kumagai, T.: Heat kernel estimates for jump processes of mixed types on metric measure spaces, Probab. Theory Related Fields {\bf 140},  277-317 (2008). 
	
	%\bibitem{CZ} Chen, Z.-Q., Zhang, T.-S.: Girsanov and Feynman--Kac type transformations for symmetric Markov processes,
	% Ann. Inst. H. Poincare Probab. Statist. {\bf 38}, 475-505 (2002).
	\bibitem{CKK} Chen, Z.-Q., Kim, D., Kuwae, K.:
	$L^p$-independence of spectral radius for generalized Feynman-Kac semigroups, to appear in Math. Ann.	
	
	
	\bibitem{CMS} Collet, P., Mart\'inez, S., San Mart\'{i}n, J.:	Quasi-stationary distributions,	Markov chains, diffusions and dynamical systems,
	Springer (2013). 
	
	%\bibitem{D1} Davies, E.B.: $L^1$ properties of second order elliptic operators, Bull. London Math. Soc. 17, 
	%417-436 (1985).
	
	
	
	
	%\bibitem{DKK}
	%De Leva, G., Kim, D., Kuwae, K.:
	% {$L^p$-independence of spectral bounds Feynman-Kac semigroups by continuous additive functionals},
	% {J. Funct. Anal.} {\bf 259}, 690--730 (2010).
	
	
	%\bibitem{Fi} Fitzsimmons, P.J.: On the quasi-regularity of semi-Dirichlet forms, Potential Anal. {\bf 15}, 
	%151-185 (2001).
	
	
	%\bibitem{Fr} Freedman, D.: Approximating Countable Markov Chains, Holden-Day (1972).
	
	%\bibitem{Fi}Fitzsimmons, P.J., Absolute continuity of symmetric diffusions, Ann. Probab. {\bf 25},
	%230-258 (1997).
	%\bibitem{FG} Fitzsimmons, P.J., Getoor, R.K. :{Revuz measures and time changes}, Math. Zeit. {\bf 199}, 233-256 (1988). 
	
	
	\bibitem{F-erg} Fukushima, M.: A note on irreducibility and ergodicity of symmetric Markov processes, 
	Springer Lecture Notes in Physics {\bf 173},  200--207 (1982).
	\bibitem{FOT} Fukushima, M., Oshima, Y., Takeda, M.: Dirichlet Forms and Symmetric Markov Processes, de 
	Gruyter, 2nd rev. and ext. ed. (2010).
	
	
	
	
	%\bibitem{G} Grigor'yan, A.: Analytic and geometric background of recurrence and non-explosion of the 
	%Brownian motion on Riemannian manifolds, Bull. Amer. Math. Soc. {\bf 36}, 135-249 (1999).
	%\bibitem{G2} Grigor'yan, A.: Heat kernels on weighted manifolds and applications, Contemp. Math. {\bf 338}, 93-191 (2006).
	%\bibitem{ping} He, P.: A formula on scattering length of dual Markov processes, 
	%Proc. Amer. Math. Soc. {\bf 139}, 1871-1877 (2011). 
	
	\bibitem{Ito} It\^o, K.: Essentials of stochastic processes, American Mathematical Society, Providence (2006). 
	
	%\bibitem{JK} Jain, N., Krylov, N.: Large deviations for occupation times of Markov processes with $L_2$ semigroups,
	% Ann. Probab. {\bf 36}, 1611-1641 (2008).
	
	%\bibitem{K1} Kac, M.: On some connections between probability theory and differential
	%equations, Proc. 2nd Berk. Symp. Math. Statist. Probability, 189-215 (1950).
	
	%\bibitem{K2} Kac, M.:{Probabilistic methods in some problems of scattering theory}, Rocky Mountain J. Math. {\bf 4 } 
	%, 511-537 (1974).
	%Scattering length for stable processes.  
	%\bibitem{KL} Kac, M.,  Luttinger, J.-M. :{Scattering length and capacity}, Ann. Inst. Fourier (Grenoble) {\bf 25}, 317-321  (1975). 
	\bibitem{K} Kajino, N.: Equivalence of recurrence and Liouville property for symmetric Dirichlet forms, 
	Mathematical Physics and Computer Simulation {\bf 40}, 89-98 (2017).
	
	\bibitem{KK} Kaleta, K., Kulczycki, T.: Intrinsic ultracontractivity for Schrödinger operators based on fractional Laplacians,
	Potential Anal. {\bf 33}, 313-339 (2010).
	
\bibitem{KT} Kulczycki, T.: Intrinsic ultracontractivity for symmetric stable processes, 
	Bull. Polish Acad. Sci. Math. {\bf 46}, 325-334 (1998). 
	
	\bibitem{KP} Knobloch, R., Partzsch, L.: Uniform conditional ergodicity and intrinsic ultracontractivity,
	Potential Anal. {\bf 33}, 107-136 (2010). 
	%\bibitem{LL} Lieb, E.H., Loss, M.: Analysis, Graduate Studies in Mathematics {\bf 14},
	% American Mathematical Society (2001).
	\bibitem{Kw} Kwa\'{s}nicki, M.: Intrinsic ultracontractivity for stable semigroups on unbounded open sets, Potential Anal. {\bf 31}, 
	57-77 (2009).
	%	\bibitem{MR} Ma, Z.M., R\"ockner, M.: Introduction to the theory of (nonsymmetric) Dirichlet forms, 
	%	Universitext, Springer-Verlag, Berlin (1992).
	%\bibitem{Ma} Maz'ja, V.G.: Sobolev Spaces, Springer (1985).
	\bibitem{miura} Miura, Y.: Ultracontractivity for Markov semigroups and quasi-stationary distributions,
	Stochastic Analysis and Applications {\bf 32}, 591-601 (2014).
	
\bibitem{Pin} Pinsky, R. G.: On the convergence of diffusion processes conditioned to remain in a bounded region for large time to limiting positive recurrent diffusion processes, Ann. Probab. {\bf 13}, 363–378 (1985).
	%\bibitem{Mu} Murata, M.: Structure of positive solutions of $(-\Delta+V)u=0$ in ${\mathbb R}^d$, Duke 
	%Math. J. 53, 869-943 (1986).
	
	%\bibitem{Sh} Sharpe, M: General theory of Markov processes, Pure and Applied Mathematics, 133. 
	%Academic Press, (1988).
 \bibitem{RW} Rogers, L. C. G., Williams, D.: Diffusions, Markov processes, and martingales, Vol. 1, Foundations, Reprint of the second edition, 
 Cambridge University Press (2000). 	
	
	
	\bibitem{T8}Takeda, M.: A tightness property of a symmetric Markov process and the uniform large deviation principle,  Proc. Amer. Math. Soc. {\bf 141}, 4371-4383 (2013).
	
	\bibitem{T-G}Takeda, M.: A variational formula for Dirichlet forms and existence of ground states, J. Funct. Anal. 
	{\bf 266}, 660-675 (2014).
	
	\bibitem{T-I}Takeda, M.: Criticality and subcriticality of generalized Schr\"{o}dinger forms,  {Illinois J. Math.}, {\bf 58}, 251-277 (2014).  
	
	\bibitem{T-C} Takeda, M.: 
	Compactness of symmetric Markov semi-groups 
	and boundedness of eigenfunctions, to appear in Trans. Amer. Math. Soc.
	
	
	\bibitem{TTT} Takeda, M., Tawara, Y., Tsuchida, K.: Compactness of Markov and Schr\"odinger semi-groups: a probabilistic approach. Osaka J. Math.
	{\bf 54}, 517-532 (2017).
	
	\bibitem{Tomi} Tomisaki, M.: Intrinsic ultracontractivity and small perturbation for one-dimensional generalized 
	diffusion operators, J. Funct. Anal. {\bf 251}, 289-324 (2007). 	 
	
	
	
\end{thebibliography}

\end{document}